\documentclass[12pt]{amsart}
\usepackage{amssymb,latexsym,amsthm}
\usepackage{graphicx}
\usepackage{color}

\newtheorem{theorem}{Theorem}[section]

\newtheorem{cor}[theorem]{Corollary}

\theoremstyle{definition}

\theoremstyle{remark}

\numberwithin{equation}{section}

\begin{document}

\baselineskip=17pt

\title[]
{Isometries between groups of invertible elements in $C^{*}$-algebras}

\author{Osamu~Hatori}
\address{Department of Mathematics, Faculty of Science,
Niigata University, Niigata 950-2181 Japan}
\curraddr{}
\email{hatori@math.sc.niigata-u.ac.jp}

\author{Keiichi~Watanabe}
\address{Department of Mathematics, Faculty of Science,
Niigata University, Niigata 950-2181 Japan}
\curraddr{}
\email{wtnbk@math.sc.niigata-u.ac.jp}

\thanks{The authors were partly
supported by the Grants-in-Aid for Scientific
Research, Japan Society for the Promotion of Science.
}


\begin{abstract}
In this  paper we describe all surjective isometries between open subgroups of the groups of invertible elements in unital $C^{*}$-algebras.
\end{abstract}
\maketitle
\section{Introduction}

This paper arises from a desire to study isometries between groups of invertible elements in unital Banach algebras (cf. \cite{hatori1,hatori2,hm,hatori3}). As is proved in \cite{hatori2} a surjective isometry between open subgroups of the groups of all invertible elements in unital semisimple Banach algebras is extended to a surjective real-linear isometry between the underlying Banach algebras. Furthermore if the given Banach algebras are commutative, then the extended map preserves multiplication if it is unital. Without assuming commutativity, we may conjecture that the Jordan structure is essentially preserved. In this paper we concentrate on the case of unital $C^{*}$-algebras and support the conjecture in this case. Due to the celebrated theorem of Kadison \cite{k} surjective complex-linear isometries between $C^{*}$-algebras are represented as a Jordan ${*}$-isomorphism followed by a multiplication by a unitary element, whereas we consider isometries as just metric spaces between open subgroups of the groups of invertible elements in unital $C^{*}$-algebras and give a complete descriptions for such isometries.

\section{The results}

Throughout this section the group of the all invertible elements in a unital $C^{*}$-algebra $A$ is denoted by $A^{-1}$. It is well-known that the group $A^{-1}$ itself is an open subset of $A$ as well as an open subgroup of $A^{-1}$ always contains the principal component of $A^{-1}$. The identity element in $A$ is denoted by $I_A$. Recall that a central projection in $A$ is a projection which is in the center of $A$.
\begin{theorem}\label{main}
Let $A$ and $B$ be unital $C^{*}$-algebras and ${\mathcal A}$ and ${\mathcal B}$ open subgroups of $A^{-1}$ and $B^{-1}$ respectively. Suppose that $T$ is a bijection from ${\mathcal A}$ onto ${\mathcal B}$. Then $T$ is an isometry if and only if $T(I_A)$ is unitary in $B$ and there are a central projection $P$ in $B$ and a surjective Jordan $*$-isomorphism $J$ from $A$ onto $B$ such that 
\begin{equation}\label{1}
T(a)=T(I_A)PJ(a)+T(I_A)(I_B-P)J(a)^*, \quad a\in {\mathcal A}
\end{equation}
holds. Furthermore the operator $T(I_A)PJ(\cdot)+T(I_A)(I_B-P)J(\cdot)^*$ defines a surjective real-linear isometry from $A$ onto $B$.
\end{theorem}
\begin{proof}
Suppose that $T:{\mathcal A}\to {\mathcal B}$ is a surjective isometry. Applying \cite[Theorem 3.2]{hatori2} $T$ is extended to a real-linear isometry $\tilde{T}$ from $A$ onto $B$ since $\mathrm{rad}(B)=\{0\}$ for $B$ is a $C^*$-algebra. A $C^*$-algebra is a real $C^*$-algebra in the sense of \cite{cdrv}, hence by \cite[Theorem 6.4]{cdrv} the equality
\begin{equation}\label{2}
\tilde{T}(ab^*c+cb^*a)=\tilde{T}(a)\tilde{T}(b)^*\tilde{T}(c)+\tilde{T}(c)\tilde{T}(b)^*\tilde{T}(a)
\end{equation}
holds for every triple $a,b,c\in A$. Letting $a=b=c=I_A$ in (\ref{2}) we have by a simple calculation that $I_B=\tilde{T}(I_A)\tilde{T}(I_A)^*=\tilde{T}(I_A)^*\tilde{T}(I_A)$ since $\tilde{T}$ is real-linear and $\tilde{T}(I_A)=T(I_A)$ is invertible in $B$; $\tilde{T}(I_A)$ is unitary in $B$. Denoting $T_0(\cdot)=(T(I_A))^{-1}\tilde{T}(\cdot)$ it requires only elementary calculation to check that $T_0$ is a surjective real-linear isometry from $A$ onto $B$ such that $T_0(I_A)=I_B$, $T_0({\mathcal A})={\mathcal B}$, and 
\begin{equation}\label{3}
T_0(ab^*c+cb^*a)=T_0(a)T_0(b)^*T_0(c)+T_0(c)T_0(b)^*T_0(a)
\end{equation}
holds for every triple $a,b,c\in A$. Substituting  $a=c=I_A$ in (\ref{3}) we have the equality
\begin{equation}\label{4}
T_0(b^*)=T_0(b)^*
\end{equation}
for every $b\in A$; $T_0$ is $*$-preserving. Then by (\ref{3}) and (\ref{4}) we obtain
\begin{equation}\label{4.5}
T_0(abc+cba)=T_0(a)T_0(b)T_0(c)+T_0(c)T_0(b)T_0(a)
\end{equation}
holds for every triple $a,b,c\in A$. Letting $a=c$ in (\ref{4.5}) observe that
\begin{equation}\label{5}
T_0(aba)=T_0(a)T_0(b)T_0(a)
\end{equation}
for every pair $a,b\in A$ by (\ref{4}). In particular, letting $b=I_A$ in (\ref{5}) we obtain
\begin{equation}\label{5.5}
T_0(a^2)=(T_0(a))^2
\end{equation}
for every $a\in A$.

We claim that $s^2=I_A$ if and only if $(T_0(s))^2=I_B$. Indeed for any $s\in A$
\begin{equation}\label{6}
s^2=I_A \Longleftrightarrow T_0(s^2)=I_B \Longleftrightarrow (T_0(s))^2=I_B.
\end{equation}
In what follow we call $s$ in $A$ (resp. $B$) a symmetry if $s=s^*$ and $s^2=I_A$ (resp. $s^2=I_B$). Then by (\ref{4}) and (\ref{6}) $s\in A$ is a symmetry if and only if $T_0(s)\in B$ is a symmetry.

Letting $a=iI_A$ in (\ref{5.5}) we obtain
\begin{equation}\label{6.5}
-I_B=T_0(-I_A)=T_0((iI_A)^2)=(T_0(iI_A))^2,
\end{equation}
hence 
\begin{equation}\label{7}
(-iT_0(iI_A))^2=I_B
\end{equation}
is observed. We also have
\begin{multline*}
T_0(iI_A)T_0(iI_A)^*=T_0(iI_A)T_0((iI_A)^*)\\
=T_0(iI_A)T_0(-iI_A)=-T_0((iI_A)^2)=T_0(I_A)=I_B
\end{multline*}
and similarly $T_0(iI_A)^*T_0(iI_A)=I_B$; $T_0(iI_A)$ is unitary. Applying that $T_0$ is real-linear
\begin{equation}\label{7.5}
(-iT_0(iI_A))^*=iT_0((iI_A)^*)=-iT_0(iI_A)
\end{equation}
is observed, hence together with (\ref{7}), $-iT_0(iI_A)$ is a symmetry in $B$. It requires a simple calculation to check that
\[
P=\frac{-iT_0(iI_A)+I_B}{2}
\]
is a projection in $B$. By (\ref{6.5})
\begin{multline*}
-T_0(a)T_0(iI_A)=T_0(iI_A)T_0(iI_A)T_0(a)T_0(iI_A) \\
= T_0(iI_A)T_0((iI_A)a(iI_A))=T_0(iI_A)T_0(-a)
\end{multline*}
holds, hence 
\begin{equation}\label{8}
T_0(a)T_0(iI_A)=T_0(iI_A)T_0(a)
\end{equation}
holds for every $a\in A$. This means that $T_0(iI_A)$ and hence $P$ commute with every elements in $B$ for $T_0({A})={B}$. Therefore $P$ is a central projection in $B$.

Define a real-linear operator $T_0':A\to B$ by $T_0'(a)=PT_0(a)+(I_B-P)T_0(a)^*$ for $a\in A$. We claim that $T_0'$ is a surjective {\it complex}-linear isometry from $A$ onto $B$. To see that $T_0'$ is a surjection, let $b\in B$ arbitrarily. Then there exists $a\in A$ with $T_0(a)=Pb+(I_B-P)b^*$ for $T_0$ is surjective and $Pb+(I_B-P)b^*\in B$. As $P$ is a central projection, so is $I_B-P$, then it can be easily calculated that $T_0'(a)=b$ holds. Hence $T_0'$ is a surjection. 

We assert that $T_0'(ia)=iT_0'(a)$ holds for every $a\in A$. This implies that $T_0'$ is complex-linear as we have already known that $T_0'$ is real-linear. Let $a\in A$. Substituting $b=iI_A$ and $c=I_A$ in the equation (\ref{4.5}) and applying (\ref{8}) we obtain
\begin{multline*}
2T_0(ia)=T_0(a(iI_A)+(iI_A)a)=T_0(a)T_0(iI_A)T_0(I_A)+T_0(I_A)T_0(iI_A)T_0(a)\\
=T_0(a)T_0(iI_A)+T_0(iI_A)T_0(a)=2T_0(iI_A)T_0(a).
\end{multline*}
Hence
\begin{equation}\label{9}
T_0(ia)=T_0(iI_A)T_0(a)
\end{equation}
holds. Then by (\ref{6.5}) 
\begin{equation}\label{9.5}
PT_0(ia)=\frac{-iT_0(iI_A)+I_B}{2}T_0(iI_A)T_0(a)=i\frac{I_B-iT_0(iI_A)}{2}T_0(a)=iPT_0(a)
\end{equation}
holds. By (\ref{8}) and (\ref{9}) we obtain that $T_0(ia)^*=T_0(iI_A)^*T_0(a)^*$. We have already learnt that $T_0$ is 
real-linear and $*$-preserving, so $T_0(iI_A)^*=-T_0(iI_A)$ 
holds, hence
\begin{equation}\label{10}
T_0(ia)^*=-T_0(iI_A)T_0(a)^*
\end{equation}
is observed. Hence we obtain
\begin{multline}\label{11}
(I_B-P)T_0(ia)^*=-\frac{I_B+iT_0(iI_A)}{2}T_0(iI_A)T_0(a)^* \\
=-\frac{T_0(iI_A)-iI_B}{2}T_0(a)^*=i(I_B-P)T_0(a)^*.
\end{multline}
It follows by (\ref{9.5}) and (\ref{11}) that $T_0'(ia)=iT_0'(a)$, observing that $T_0'$ is complex-linear. 

Next we assert that $T_0'$ is an isometry. For this purpose we will verify that 
\[
\|Pb+(I_B-P)c\|=\max\{\|Pb\|,\|(I_B-P)c\|\}
\]
holds for every pair $b$ and $c$ in $B$. We may suppose that $B$ is a subalgebra of the algebra $B(H)$ of all bounded linear operators on a Hilbert space $H$. Let $x$ be an arbitrary element in $H$. We have learnt that $P$ is a central projection, and so is $I_B-P$, hence $Pbx=PbPx$ and $(I_B-P)cx=(I_B-P)c(I_B-P)x$ hold, and 
\begin{multline*}
\|(Pb+(I_B-P)c)x\|^2=\|PbPx\|^2+\|(I_B-P)c(I_B-P)x\|^2
\\
\le \|Pb\|^2\|Px\|^2+\|(I_B-P)c\|^2\|(I_B-P)x\|^2
\\
\le \max\{\|Pb\|^2,\|(I_B-P)c\|^2\}(\|Px\|^2+\|(I_B-P)x\|^2)
\\
=\max\{\|Pb\|^2,\|(I_B-P)c\|^2\}\|x\|^2
\end{multline*}
hold, asserting that 
\[
\|Pb+(I_B-P)c\|\le \max \{\|Pb\|, \|(I_B-P)c\|\}.
\]
By a simple calculation we can easily check that
\[
\|Pbx\|,\,\,\|(I_B-P)cx\|\le \|Pbx+(I_B-P)cx\|\le \|Pb+(I_B-P)c\|\|x\|
\]
hold for every $x\in H$, hence 
\[
\max \{\|Pb\|, \|(I_B-P)c\|\}\le \|Pb+(I_B-P)c\|
\]
is obtained. It follows that 
\[
\max \{\|Pb\|, \|(I_B-P)c\|\}= \|Pb+(I_B-P)c\|
\]
holds. Applying the above equation for $b=T_0(a)$ and $c=T_0(a)^*$ or $T_0(a)$ we can easily check that 
\begin{multline*}
\|T_0'(a)\|=\|PT_0(a)+(I_B-P)T_0(a)^*\|=\max\{\|PT(a)\|, \|(I_B-P)T_0(a)^*\|\} \\
=\max\{\|PT_0(a)\|, \|(I_B-P)T_0(a)\|\}=\|T_0(a)\|=\|a\|
\end{multline*}
hold for every $a\in A$, observing that  $T_0'$ is an isometry as $T_0'$ is real-linear. 

We have obtained that $T_0'$ is a surjective isometry from $A$ onto $B$ such that $T_0'(I_A)=I_B$. By just a simple application of a celebrated representation theorem of Kadison for isometries between $C^*$-algebras, $T_0'$ is a surjective Jordan $*$-isomorphism from $A$ onto $B$. Denoting $J=T_0'$, $J(a)=PT_0(a)+(I_B-P)T_0(a)^*$ for every $a\in A$. Then $PJ(a)=PT_0(a)$, and 
\[
(I_B-P)J(a)^*=((I_B-P)J(a))^*=((I_B-P)T_0(a)^*)^*=(I_B-P)T_0(a)
\]
hold since $I_B-P$ is self-adjoint and is in the center of $B$. Thus $T_0(a)=PJ(a)+(I_B-P)J(a)^*$ holds for every $a\in A$. As $T(I_A)$ is unitary, $T(I_A)T_0(\cdot)$ is a surjective isometry form $A$ onto $B$ and (\ref{1}) holds for every $a \in {\mathcal A}$. 

Conversely, suppose that $T(I_A)$ is unitary in $B$, $P$ is a central projection in $B$, and $J$ is a surjective Jordan $*$-isomorphism from $A$ onto $B$ such that
\[
T(a)=T(I_A)PJ(a)+T(I_A)(I_B-P)J(a)^*
\]
holds for every $a\in {\mathcal A}$. Since $J$ is an isometry from $A$ onto $B$ (cf. \cite[Theorem 6.2.5]{fj}) it requires only a way similar to the above argument to verify that $T(I_A)PJ(\cdot)+T(I_A)(I_B-P)J(\cdot)^*$ defines a surjective isometry from $A$ onto $B$. As $T$ is a bijection from ${\mathcal A}$ onto ${\mathcal B}$, we obtain that $T$ is an isometry from ${\mathcal A}$ onto ${\mathcal B}$.
\end{proof}

For a complex Hilbert space $H$ the algebra of all bounded linear operators on $H$ is denoted by $B(H)$.
We describe the structure of all surjective isometries between open subgroups of $B(H)^{-1}$. To formulate the result we need the following notation. Beside the adjoint operation on the algebra $B(H)$ we shall also need the operation of transposition. It is defined by choosing a complete orthonomal system in $H$ and for any operator $a$ considering the operator $a^T$ whose matrix in the given basis is the transpose of the corresponding matrix of $a$. It can be seen that the map $a\mapsto a^T$ is a well-defined linear $*$-antiautomorphism of $B(H)$. Then conjugate $\overline{a}$ of $a\in B(H)$ is defined by the the formula $\overline{a}=(a^{*})^T$. Our result reads as follows.
\begin{cor}
Let $H_1$ and $H_2$ be complex Hilbert spaces. Suppose that $T$ is a surjective isometry from ${\mathcal A}$ onto ${\mathcal B}$, where ${\mathcal A}$ and ${\mathcal B}$ are open subgroups of $B(H_1)^{-1}$ and $B(H_2)^{-1}$ respectively. Then $T(I_{B(H_1)})$ is a unitary operator and there is a unitary operator $w$ (complex-linear isometry) from $H_1$ onto $H_2$ such that $T$ is of one of the following forms:
\begin{enumerate}
\item
$T(a)=T(I_{B(H_1)})waw^{*}\quad \text{for all}\,\, a\in {\mathcal A}$

\item
$T(a)=T(I_{B(H_1)})wa^{*}w^{*}\quad \text{for all}\,\, a\in {\mathcal A}$

\item
$T(a)=T(I_{B(H_1)})wa^{T}w^{*}\quad \text{for all}\,\, a\in {\mathcal A}$

\item
$T(a)=T(I_{B(H_1)})w\overline{a}w^{*}\quad \text{for all}\,\, a\in {\mathcal A}$
\end{enumerate}
\end{cor}
\begin{proof}
The center of $B(H_2)$ is the scalar, hence $P$ is a trivial projection ($0$ operator or $I_{B(H_2)}$). Then Theorem \ref{main} asserts that there is a Jordan ${*}$-isomorphism $J$ and $T(a)=T(I_{B(H_1)})J(a)$ for all $a\in {\mathcal A}$ or $T(a)=T(I_{B(H_1)})J(a)^{*}$ for all $a\in {\mathcal A}$. A Jordan $*$-isomorphism from $B(H_1)$ onto $B(H_2)$ is an algebra isomorphism or an algebra antiisomorphism, hence there is a unitary operator $w$ from $H_1$ onto $H_2$ such that $J(b)=wbw^{*}$ for every $b\in B(H_2)$ or $J(b)=w{b}^Tw^{*}$ for every $b\in B(H_2)$ holds. Thus we have the conclusion. 
\end{proof}


\end{document}